\documentclass[notitlepage]{scrartcl}
\usepackage{enumitem}
\usepackage{amsmath}
\usepackage{amssymb}
\usepackage{mathtools}

\usepackage{amsthm}
\usepackage{abstract}
\usepackage{xcolor}
\usepackage{comment}  
\usepackage{dsfont}
\usepackage{hyperref}
\usepackage{amsmath}

\usepackage{algorithm}
\usepackage{algpseudocode}
\hypersetup{
    colorlinks=true,
    linkcolor=blue,
    filecolor=magenta,      
    urlcolor=cyan
    }
\makeatletter
\newcommand*{\barfix}[2][.175ex]{%
  \mathpalette{\@barfix{#1}}{#2}%
}
\newcommand*{\@barfix}[3]{%
  \vbox{%
    \kern#1\relax
    \hbox{$#2#3\m@th$}%
  }%
}
\newcommand{\linebreakand}{%
  \end{@IEEEauthorhalign}
  \hfill\mbox{}\par
  \mbox{}\hfill\begin{@IEEEauthorhalign}
}
\makeatother

\newtheorem{thm}{Theorem}[section]

\newtheorem{claim}[thm]{Claim}

\theoremstyle{remark}
\newtheorem{remark}{Remark}

\usepackage{setspace}
\usepackage[margin=1in]{geometry}

\title{On vertex Ramsey graphs with forbidden subgraphs}

\author{Sahar Diskin \qquad Ilay Hoshen \\
        Michael Krivelevich\footnote{Research supported in part by USA-Israel BSF grant 2018267.} \qquad Maksim Zhukovskii\footnote{Research supported in part by Israel ISF grant 2110/22.}\\
        School of Mathematical Sciences, Tel Aviv University
}

\date{}
\begin{document}

\maketitle
\begin{abstract}
A classical vertex Ramsey result due to Ne\v{s}et\v{r}il and R\"odl states that given a finite family of graphs $\mathcal{F}$, a graph $A$ and a positive integer $r$, if every graph $B\in\mathcal{F}$ has a $2$-vertex-connected subgraph which is not a subgraph of $A$, then there exists an $\mathcal{F}$-free graph which is vertex $r$-Ramsey with respect to $A$. We prove that this sufficient condition for the existence of an $\mathcal{F}$-free graph which is vertex $r$-Ramsey with respect to $A$ is also \textit{necessary} for large enough number of colours $r$.

We further show a generalisation of the result to a family of graphs and the typical existence of such a subgraph in a dense binomial random graph.
\end{abstract}

\section{Introduction}

Let $A$ be a graph and let $r$ be a positive integer. We say that a graph $G$ is \emph{(vertex) $r$-Ramsey} with respect to $A$ if in every colouring of the vertices of $G$ in $r$ colours there exists a monochromatic copy of $A$. The existence of $r$-Ramsey graphs is straightforward: the complete graph $K_n$ is $r$-Ramsey with respect to $A$ for every $n\geq r(|V(A)|-1)+1$. It is thus natural to ask about the existence of \textit{sparse} Ramsey graphs. One of the ways to define sparseness is to avoid copies of a given graph $B$ (or more generally of any graph from a given finite graph family $\mathcal{F}$) in $G$. Let us call a graph $G$ \emph{$\mathcal{F}$-free} if it does not contain a subgraph isomorphic to $B$ for every $B\in\mathcal{F}$.

Perhaps the most studied case is when both $A$ and $B$ are complete graphs on $s$ and $t$ vertices, respectively, where $t>s\ge 2$. Denote by $f_{s,t}(n)$ the minimum over all $K_t$-free graphs $G$ on $[n]:=\{1,\ldots,n\}$ of the maximum number of vertices in an induced $K_s$-free subgraph of $G$. 
Erd\H{o}s and Rogers \cite{ER} proved that, for a certain $\varepsilon=\varepsilon(s)>0$, $f_{s,s+1}(n)\leq n^{1-\varepsilon}$ (note that this implies that for every $s\geq 2$ and $r\geq 2$, there exists a $K_{s+1}$-free graph $G$ which is $r$-Ramsey with respect to $K_s$). The result of Erd\H{o}s and Rogers was subsequently refined by Bollob\'{a}s and Hind~\cite{BH} and Krivelevich \cite{Krivelevich}. Let us also mention that subsequent works by Dudek, Retter and R\"{o}dl \cite{Dudek} and by Dudek and R\"{o}dl \cite{DR} determined $f_{s,s+1}(n)$ up to a power of $\log n$ factor, strengthened the known bounds for $f_{s,s+2}(n)$, and further improved the bounds for $f_{s,s+k}(n)$ when $s,k$ are large enough.

Considering general graphs $A$ and $B$ (and in fact, a family of graphs $B$), Ne\v{s}et\v{r}il and R\"odl \cite{NR vertex} proved the following (see also \cite{BW}): 
\begin{thm}[\cite{NR vertex}]
Let $\mathcal{F}$ be a finite family of graphs and let $A$ be a graph. Let $r\geq 2$ be an integer. If every graph from $\mathcal{F}$ has a $2$-vertex-connected subgraph which is not a subgraph of $A$, then there exists an $\mathcal{F}$-free graph which is vertex $r$-Ramsey with respect to $A$.
\label{th:sufficient}
\end{thm}
See \cite{RS92, RSZ95, SZ91} for additional results on vertex-Ramsey graphs with forbidden subgraphs.

Our main result shows that the above sufficient condition is also necessary for large enough number of colours $r$. We say that $B$ is an $A$-forest of size $\ell$ if $B=\cup_{i=1}^{\ell}B_i$, where for every $1\le i\le \ell$, $B_i$ is isomorphic to a subgraph of $A$, and for every $i\ge 2$, $\big|V(B_i)\cap V\left(\cup_{j=1}^{i-1}B_j\right)\big|\le 1$.
\begin{thm}\label{th:main}
Let $\ell>0$ be an integer. Let $B$ be an $A$-forest of size $\ell$. Let $r>0$ be an integer such that $r\ge \ell\left(2(|V(A)|-1)(|V(B)|-2)+1\right)$, and let $G$ be an $r$-Ramsey graph with respect to $A$. Then $G$ contains a copy of $B$.
\end{thm}
Let us first note that since $\ell\le |V(B)|$, it suffices to take $r=O(|V(A)||V(B)|^2)$. Furthermore, observe that the above implies the necessity of the condition in Theorem \ref{th:sufficient}, for $r$ large enough. Indeed, let us say that a graph $B$ is \emph{$A$-degenerate}, if every $2$-vertex-connected subgraph of it is a subgraph of $A$. Note that any $A$-degenerate graph can be constructed recursively: (1) any subgraph of $A$ is $A$-degenerate; (2) if $B$ is an $A$-degenerate graph, then a union of $B$ with a subgraph of $A$ that shares with $B$ at most 1 vertex is $A$-degenerate as well. Theorems \ref{th:sufficient} and \ref{th:main} can be formulated in terms of $A$-degenerate graphs: there exists an $\mathcal{F}$-free graph which is $r$-Ramsey with respect to $A$ for all large enough $r$ \textit{if and only if} every graph from $\mathcal{F}$ is not $A$-degenerate.

Note that the case that $B$ consists of $\ell$ vertex-disjoint components, each isomporphic to a subgraph of $A$, is easy since if $G$ is $r$-Ramsey with respect to $A$ then it contains a large enough family of vertex-disjoint copies of $A$. On the other hand, if the components of $B$ are not disjoint, we can proceed by induction, deleting a component $B_i$ intersecting other components, finding a copy of $B-B_i$ using inductive hypothesis and then adjoining to it a correctly placed copy of $B_i$, see details in Section \ref{section:proofs-section}. \\

In the next section, we provide a short proof of Theorem \ref{th:sufficient} for the sake of completeness, followed by the proof of Theorem \ref{th:main}. In Section \ref{remarks}, we discuss generalisations of Theorem \ref{th:sufficient} to a family of graphs (instead of $A$), and the existence of an $\mathcal{F}$-free graph which is $r$-Ramsey with respect to $A$ in a dense enough binomial random graph.

\section{Proofs of Theorems \ref{th:sufficient} and \ref{th:main}}\label{section:proofs-section}
We say that a graph $G$ is \emph{$\varepsilon$-dense} with respect to a graph $A$ if every induced subgraph of $G$ on $\lfloor\varepsilon|V(G)|\rfloor$ vertices contains a copy of $A$. Clearly, if $G$ is $1/r$-dense with respect to $A$, then it is also $r$-Ramsey with respect to $A$. Theorem~\ref{th:sufficient} follows immediately from Theorem \ref{th:dense_1}.
\begin{thm}
Let $\mathcal{F}$ be a finite family of graphs. If there are no $A$-degenerate graphs in $\mathcal{F}$, then there exists a $\delta=\delta(A,\mathcal{F})>0$ such that for all large enough $n$, there exists an $\mathcal{F}$-free $n^{-\delta}$-dense graph on $[n]$ with respect to $A$.
\label{th:dense_1}
\end{thm}
\begin{proof}
Let $a\coloneqq |V(A)|$. Let $\varepsilon>0$ be small enough and set $p=n^{1-a+\varepsilon}$. Consider a hypergraph with vertex set $\binom{[n]}{2}$ whose edge set consists of all possible copies of $A$ on $[n]$. Let $\mathcal{H}_A(n,p)$ be its binomial subhypergraph where each copy of $A$ is chosen independently and with probability $p$, and let $\mathcal{G}_A(n,p)$ be the random graph constructed as follows: an edge belongs to $\mathcal{G}_A(n,p)$ if and only if this edge belongs to a copy of $A$ in $\mathcal{H}_A(n,p)$. We shall prove that it suffices to remove $O(\sqrt{n})$ vertices of $\mathcal{G}_A(n,p)$ to get the desired graph \textbf{whp}.

Let $\delta_0=\frac{\varepsilon}{2(a-1)}$. Let us show that \textbf{whp} $\mathcal{G}_A(n,p)$ is $n^{-\delta_0}$-dense with respect to $A$. Set $N=\lfloor n^{1-\delta_0}\rfloor$. Then the expected number of $N$-sets containing no copy of $A$ in $\mathcal{G}_A(n,p)$ is at most the expected number of $N$-subsets $U \subseteq [n]$ such that $\binom{U}{2}$ does not contain any copy of $A$ in $\mathcal{H}_A(n,p)$ that equals to 
\begin{align*}
 \binom{n}{N}(1-p)^{\binom{N}{a}\frac{a!}{\mathrm{aut}(A)}} &\leq
 \exp\left[N\left(\delta_0\ln n+1-p\frac{N^{a-1}}{\mathrm{aut}(A)}\right)(1+o(1))\right]\\
 &\leq
 \exp\left[N\left(\delta_0\ln n-\frac{n^{1-a+\varepsilon+(a-1)(1-\delta_0)}}{\mathrm{aut}(A)}\right)(1+o(1))\right]\\
 &\leq \exp\left[-Nn^{\varepsilon/2}\left(\frac{1}{\mathrm{aut}(A)}-o(1)\right)\right]
  \to 0.
\end{align*}
By the union bound, \textbf{whp} every $N$-set contains at least one copy of $A$ in $\mathcal{G}_A(n,p)$, that is, \textbf{whp} $\mathcal{G}_A(n,p)$ is $n^{-\delta_0}$-dense. 

Let $\delta=\delta_0/2$ and let $C>0$. Note that \textbf{whp} the deletion of any $C\sqrt{n}$ vertices from $\mathcal{G}_A(n,p)$ leads to an $\tilde n^{-\delta}$-dense graph on $\tilde n$ vertices. Indeed, if $\mathcal{G}_A(n,p)$ is $n^{-\delta_0}$-dense, then, since $\tilde{n}^{1 - \delta} = (n - C\sqrt{n})^{1 - 0.5\delta_0} \ge n^{1-\delta_0}$, every set of $\tilde{n}^{1 - \delta}$ vertices in the new graph has at least $n^{1-\delta_0}$ vertices and thus contains a copy of $A$. Therefore, it suffices to prove that \textbf{whp} we can remove $O(\sqrt{n})$ vertices from $\mathcal{G}_A(n,p)$ and get an $\mathcal{F}$-free graph.

Given a graph $B$ and graphs $A_1, \dots, A_m$ isomorphic to $A$, we say that $A_1\cup\ldots\cup A_m$ is an \textit{inclusion-minimal cover} of the edges of $B$ if $E(B) \subseteq E(A_1 \cup \ldots \cup A_m)$ but $E(B) \not\subseteq E(A_1 \cup \ldots \cup A_{i-1} \cup A_{i+1} \cup \ldots \cup A_m)$ for every $i \in [m]$. For every $B\in\mathcal{F}$, consider $B'\subset B$ such that every inclusion-minimal cover $A_1\cup\ldots\cup A_m$ of the edges of $B'$ satisfies $|(A_i\cap\cup_{j\neq i}A_j)\cap B'|\geq 2$ for every $i\in[m]$. By Claim~\ref{cl:main} (stated below), \textbf{whp} the number of copies of $B'$ in $\mathcal{G}_A(n,p)$ is at most $\sqrt{n}$. We can now delete a single vertex from each such copy, and obtain a set of $\tilde n\geq n-|\mathcal{F}|\sqrt{n}$ vertices that induces an $\mathcal{F}$-free graph, as required.
\end{proof}

We note that a slight adjustment of the proof of Theorem \ref{th:dense_1} allows one to argue for the existence of $\mathcal{F}$-free $\varepsilon$-dense graph for \textit{induced} copies of $A$.

\begin{claim}
\textbf{Whp} the number of copies of $B'$ in $\mathcal{G}_A(n,p)$ is at most $\sqrt{n}$.
\label{cl:main}
\end{claim}
\begin{proof}
Let $b\coloneqq |V(B')|$ and let $k\coloneqq |E(B')|$. Let $X_{B'}$ be the number of copies of $B'$ in $\mathcal{G}_A(n,p)$. We shall bound $\mathbb{E}X_{B'}$ from above.

A copy of $B'$ may appear in $\mathcal{G}_A(n,p)$ only through hyperedges $A_1,\ldots,A_{\ell}\in E(\mathcal{H}_A(n,p))$ such that $B'\subset A_1\cup\ldots\cup A_{\ell}$. For any possible inclusion-minimal cover of edges of $B'$ by copies $A_1,\ldots,A_{\ell}$ of $A$, let us denote by $v_i$ the number of vertices in the intersection of $A_i$ and $B'$. Then, each $A_i$ in this cover contributes a factor of $O(n^{a-v_i}p)$ to $\mathbb{E}X_{B'}$. More formally, if $B'=B_1\cup\ldots\cup B_{\ell}$, where each $B_i$ is a subgraph of a copy of $A$, then, for every $i$,
$$
 \mathbb{P}(\exists A'\in\mathcal{H}_A(n,p)\colon A'\supset B_i)=O(n^{a-v_i}p).
$$
Since there are $O(n^b)$ choices of $B'$ in $K_n$, we get that
\begin{align}
 \mathbb{E}X_{B'} & =O\left(n^b\max_{A_1\cup\ldots\cup A_{\ell}\supset B'}n^{a\ell-v_1-\ldots-v_{\ell}}p^{\ell}\right)\notag\\
 &=O\left(n^{b+\max_{A_1\cup\ldots\cup A_{\ell}}(\ell(1+\varepsilon) - v_1 - \ldots - v_{\ell})}\right),
\label{eq:expectation_B}
\end{align}
where the maximum and minimum are taken over all inclusion-minimal covers $A_1\cup\ldots\cup A_{\ell}$ of edges of $B'$ by copies of $A$.

Let $A_1\cup\ldots\cup A_{\ell}$ be an inclusion-minimal cover of the edges of $B'$ by copies of $A$, and let $V_i$ be the set of vertices in the intersection of $A_i$ with $B'$ (as above, we let $v_i=|V_i|$). Since each $V_i$ has at least two common vertices with $\cup_{j\neq i} V_j$ and $\ell\geq 2$, we get
\begin{equation}
\sum_{i=1}^{\ell}v_i\ge |V_1\cup \ldots \cup V_{\ell}|+\ell=b+\ell,
\label{eq:bound_sums_cover}
\end{equation} 
that is $\sum_{i=1}^{\ell}v_i\geq b+\ell$. Indeed, for every $i$, let $S_i=V_i\cap(\cup_{j\neq i}V_j)$, $s_i=|S_i|\geq 2$. Then $V_1\cup\ldots\cup V_{\ell}=S_1\cup\ldots\cup S_{\ell}\cup\Sigma$, where $\Sigma$ is the set of vertices that are covered once. Then $|\Sigma|=\sum_{i=1}^{\ell}(v_i-s_i),$ and $|S_1\cup\ldots\cup S_{\ell}|\leq\frac{1}{2}\sum_{i=1}^{\ell}s_i$, since each vertex in this union is covered at least twice. We thus obtain,
$$
|V_1\cup\ldots\cup V_{\ell}|=|\Sigma|+|S_1\cup\ldots\cup S_{\ell}|\leq\sum_{i=1}^{\ell} v_i-\frac{1}{2}\sum_{i=1}^{\ell} s_i\leq \sum_{i=1}^{\ell} v_i-\ell,
$$ 
where the last inequality follows since each $s_i$ is at least 2. 

We may assume that $\varepsilon<\frac{1}{2k}$. Due to (\ref{eq:expectation_B}) and (\ref{eq:bound_sums_cover}), we get
$$
 \mathbb{E}X_{B'}=O(n^{k\varepsilon})=o(\sqrt{n}).
$$
By Markov's inequality, \textbf{whp} we have less than $\sqrt{n}$ copies of $B'$ in $\mathcal{G}_A(n,p)$.
\end{proof}

We now turn to the proof of our main theorem. 

\begin{proof}[Proof of Theorem \ref{th:main}]
Let $a\coloneqq|V(A)|$ and $b\coloneqq|V(B)|$. If $a=1$, that is, $A=K_1$, then note that $B$ is the empty graph on $b$ vertices and thus every graph on at least $b$ vertices contains a copy of $B$. If $a = b = 2$, then we have $B \subseteq A$ and thus every graph which is $r$-Ramsey with respect to $A$ contains a copy of $B$.

We assume that $a\ge 2, b\ge 3$. We enumerate the vertices of $A$: $V(A)=\{v_1,\ldots,v_a\}$. Given a graph $G$, and a copy $A'$ of $A$ in $G$, we define a mapping $\phi_{A'}:V(A)\to V(G)$ such that for every $v_i\in V(A)$, we set $\phi_{A'}(v_i)$ to be the vertex $v\in V(G)$ which is in the role of $v_i$ in the copy $A'$ of $A$. Given $v\in V(G)$, we denote by $\mathcal{A}_i(v)$ the set of copies $A'$ of $A$ in $G$ for which $\phi_{A'}(v_i)=v$. Furthermore, we denote by $s_i(v)$ the maximal size of a subset of $\mathcal{A}_i(v)$, in which every two copies of $A$ in $G$ intersect only at $v$.

We prove by induction on $\ell$, the minimum size of an $A$-forest of $B$, where the base case $\ell=1$ is trivial.

We now consider two cases separately. First, assume that in an $A$-forest of $B$ of size $\ell$ all components $B_i$ are disjoint. Let $M$ be the maximum size of a family of vertex disjoint copies of $A$ in $G$. Then, we can colour each copy in a maximum family of vertex disjoint copies of $A$ in two separate colours, and colour all the other vertices in $(2M+1)$-th colour, without producing a monochromatic copy $A$. As $G$ is $r$-Ramsey with respect to $A$, we conclude that $r<2M+1$. Since $r\ge 2\ell$, we find $\ell$ disjoint copies of $A$ in $G$, and therefore a copy of $B$ in $G$.

We now turn to the case where, without loss of generality, $B_{\ell}$ intersects $\cup_{i=1}^{\ell-1}B_i$ in an $A$-forest of $B$. Let $\tilde{B}\coloneqq\cup_{i=1}^{\ell-1}B_i$, and let $\{x\}\coloneqq V(\tilde{B})\cap V(B_{\ell})$. We may further assume that for $A\supseteq B_{\ell}$, $x$ corresponds to $v_k$ in $A$, for some $1\le k \le a$.

Let $U=\left\{v\in V(G)\colon s_k(v)\le b-2\right\}$. We require the following claim.
\begin{claim}
$G[U]$ can be coloured in $2(a-1)(b-2)+1$ colours, without a monochromatic copy of $A$.
\end{claim}
\begin{proof}
For every $v\in U$, let $\mathcal{S}_k(v)$ be a maximal by inclusion subfamily of $\mathcal{A}_k(v)$ composed of copies of $A$ in $G[U]$, where every two copies of $A$ in the subfamily intersect only at $v$, and let $S_k(v)=\cup_{A'\in \mathcal{S}_k(v)}V(A')$. By definition of $U$, $|\mathcal{S}_k(v)|\le b-2$ and $|S_k(v)|\le (a-1)(b-2)+1$. 

Define an auxiliary directed graph $\overrightarrow{\Gamma}$ on the vertices of $U$, where for every $v$ and for every $u\in S_k(v)\setminus\{v\}$, $\overrightarrow{\Gamma}$ contains a directed edge from $v$ to $u$. We thus have that $\Delta^{+}(\overrightarrow{\Gamma})\le (a-1)(b-2)$. Hence, the underlying undirected graph $\Gamma$ is $2(a-1)(b-2)$-degenerate. Indeed, consider $V' \subseteq V(\Gamma)$. We will show that in the induced subgraph $\Gamma[V']$ there exists a vertex of degree at most $2(a-1)(b-2)$. We have
\[
    \sum_{v \in V'} d_{\Gamma[V']}(v) = 2 |E(\Gamma[V'])| \le 2 \sum_{v \in V'} d^{+}_{\overrightarrow{\Gamma}}(v) \le 2(a-1)(b-2) |V'|,
\]
and thus there must be at least one vertex $v \in V'$ with $d_{\Gamma[V']}(v) \le 2(a-1)(b-2)$. Therefore, $\Gamma$ is $\left(2(a-1)(b-2)+1\right)$-colourable. We colour $G[U]$ according to this colouring.

Suppose towards contradiction that there is a monochromatic copy $A'$ of $A$ in $G[U]$, and let $w=\phi_{A'}(v_k)$. Since $A'$ is monochromatic, it does not have common vertices with $S_k(w)$ other than $w$ --- however this contradicts the maximality of $\mathcal{S}_k(w)$.
\end{proof}

Recalling that $G$ is $r$-Ramsey with respect to $A$, and that $G[U]$ can be coloured in $2(a-1)(b-2)+1$ colours without containing a monochromatic copy of $A$, we have that $G[V\setminus U]$ is $\left(r-\left(2(a-1)(b-2)+1\right)\right)$-Ramsey with respect to $A$. Observing that $$r-\left(2(a-1)(b-2)+1\right)\ge (\ell-1)\left(2(a-1)(b-2)+1\right),$$
we have by induction that $G[V\setminus U]$ contains a copy of $\tilde{B}$. Let $v$ be the vertex in this copy of $\tilde{B}$ that corresponds to $x$. Since $v\notin U$ we have that $s_k(v)\ge b-1$, and hence there is a subset of size at least $b-1$ in $\mathcal{A}_k(v)$ such that every two copies $A'$ of $A$ in this subset intersect only at $v$. Noting that $|V(\tilde{B})|\le b-1$, we have that at least one copy $A'$ of $A$ in this subset completes $\tilde{B}$ to $B$, that is, $\tilde{B}\cup A'$ contains a copy of $B$ (see Figure \ref{f:one} for an illustration).
\end{proof}
\begin{figure}[H]
\centering
\includegraphics[width=0.75\textwidth]{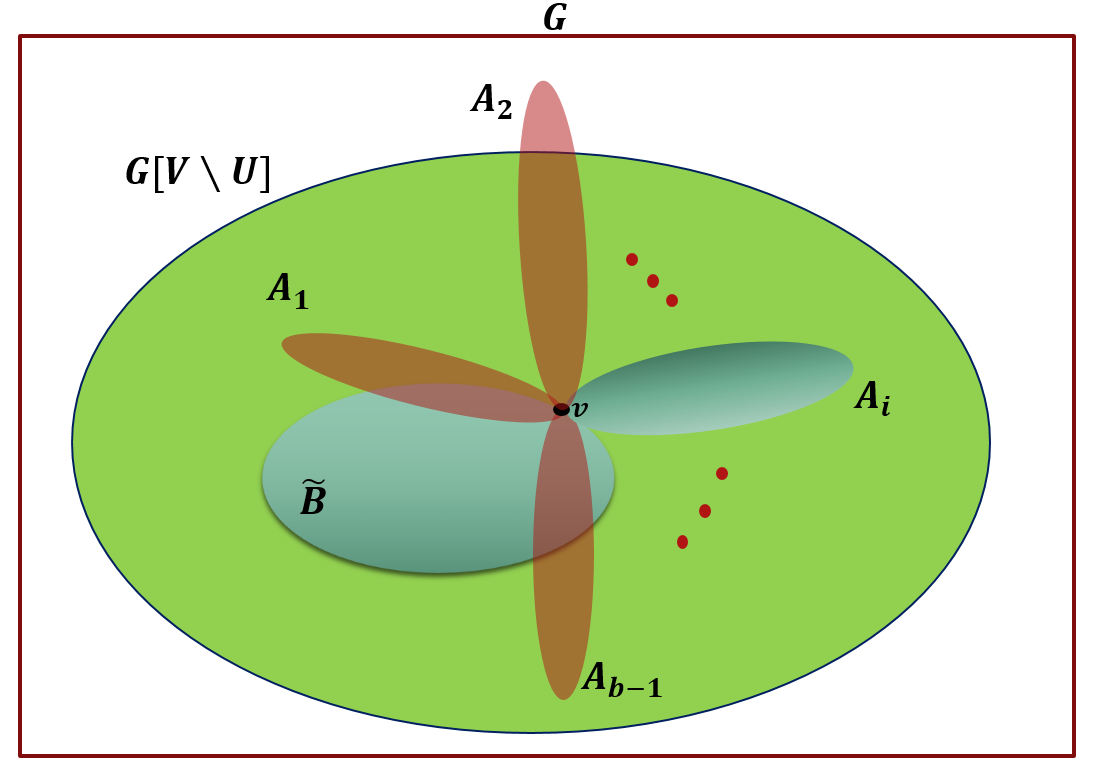}
\caption{The subgraph $G[V\setminus U]$ and a copy of $\tilde B$ in it. A copy $A_i$ of $A$ together with $\tilde{B}$ contain a copy of $B$. Note that some of the $A_j$'s may have vertices outside $V\setminus U$.}
\label{f:one}
\end{figure}

\section{Remarks and observations}\label{remarks}
Let us finish with two remarks.

\begin{remark}
Theorems~\ref{th:dense_1} and \ref{th:main} can be generalised to families of graphs instead of a single graph $A$. Let $\mathcal{A},\mathcal{F}$ be two finite graph families, and $\varepsilon>0$. The proof of Theorem \ref{th:main} is quite similar. For the proof of Theorem \ref{th:dense_1}, let us say that a graph $G$ is an $\mathcal{F}$-free \emph{$\varepsilon$-dense} with respect to $\mathcal{A}$ if it is $B$-free for every $B\in\mathcal{F}$, and every induced subgraph of $G$ on exactly $\lfloor\varepsilon|V(G)|\rfloor$ vertices contains a copy of every $A\in\mathcal{A}$. A graph $B$ is \emph{$\mathcal{A}$-degenerate}, if every $2$-vertex-connected subgraph of it is isomorphic to a subgraph of some $A\in\mathcal{A}$. If every $B\in\mathcal{F}$ is not $\mathcal{A}$-degenerate, then there exists an $\mathcal{F}$-free $\varepsilon$-dense graph with respect to $\mathcal{A}$ --- indeed, let $A$ be the disjoint union of the graphs from $\mathcal{A}$, and apply Theorem \ref{th:dense_1}. 
\end{remark}

\begin{remark}
For a non-$A$-degenerate family $\mathcal{F}$ (consisting of graphs that are not $A$-degenerate) and sufficiently small $\delta>0$, we claim the likely existence of an $\mathcal{F}$-free $n^{-\delta}$-dense subgraph in the binomial random graph $G(n,n^{-2/a+\delta})$, where $a$ is the number of vertices in $A$. Indeed, consider the hypergraph with vertex set $\binom{[n]}{2}$, and edge set being all the possible cliques of size $a$, $K_a$, on $[n]$. Let $\mathcal{H}_a(n,p')$ be its binomial subgraph. Let us first show that there exists a coupling between $\mathcal{H}_{a}(n,p')$, and the graph considered in the proof of sufficiency of Theorem~\ref{th:dense_1}, $\mathcal{H}_{A}(n,p)$, such that $p=\Theta(p')$ and $\mathcal{H}_A(n,p)\subseteq \mathcal{H}_a(n,p')$. Indeed, let $p=n^{1-a+\delta{\binom{a}{2}}}$. Consider an $a$-set, and let $p'$ be the probability that at least one copy of $A$ appears on this $a$-set. Clearly, $p=\Theta(p')$. Let $Q$ be the conditional distribution of a binomial random hypergraph of copies of $A$ on $[a]$, under the condition that at least one such copy exists. We can now draw $\mathcal{H}_{A}(n,p)$ as follows. We first choose every $a$-set with probability $p'$, and then in every set that we chose we construct a random $A$-hypergraph with distribution $Q$, independently for different $a$-sets. We thus have that $\mathcal{H}_A(n,p)\subseteq \mathcal{H}_a(n,p')$, and we can continue the proof in the same manner as in Theorem \ref{th:dense_1}. Now, we take $q$ such that $q^{\binom{a}{2}}= p'$. Therefore, by the above coupling between $\mathcal{H}_a(n,p')$ and $\mathcal{H}_A(n,p)$ and by Theorem \ref{riordan} stated below, \textbf{whp} $G(n,q)\supset \mathcal{G}_{K_a}(n,p')\supset \mathcal{G}_A(n,p)$.
\end{remark}
\begin{thm}[Riordan \cite{Riordan}]\label{riordan}
Let $\varepsilon>0$ be small enough and $q\leq n^{-\frac{2}{a}+\varepsilon}$, $p\sim q^{{\binom{a}{2}}}$. Then there exists a coupling between $G(n,q)$ and $\mathcal{H}_a(n,p)$ such that \textbf{whp} for every edge of $\mathcal{H}_a(n,p)$ there exists a copy of $K_a$ in $G(n,q)$ with the same vertex set.
\end{thm}
We note that Riordan in \cite[Section 5]{Riordan} discusses a coupling between $G(n,q)$ and $\mathcal{H}_A(n,p)$, and provides sufficient conditions for its existence for some $A$, however here we settle for higher values of $q(n)$ with respect to $p(n)$, thus making such coupling simpler.

\paragraph{Acknowledgement.} The authors wish to thank Benny Sudakov for helpful remarks.

\end{document}